\documentclass{amsart}

\usepackage{amssymb,amsmath,amsthm,latexsym,stmaryrd,xcolor}

\newtheorem{theorem}{Theorem}[section]

\newtheorem{corollary}[theorem]{Corollary}

\newtheorem{definition}[theorem]{Definition}

\newcommand{\ie}{i.~\!e.\ }
\newcommand{\f}{\varphi}
\newcommand{\vt}{\vartheta}
\newcommand{\g}{\tilde{g}}
\newcommand{\bg}{\bar{g}}
\newcommand{\tbg}{\tilde{\bar{g}}}
\newcommand{\n}{\nabla}

\newcommand{\M}{(M,\A\f,\A\xi,\A\eta,\A{}g)}

\newcommand{\I}{\iota}

\newcommand{\R}{\mathbb R}

\newcommand{\F}{\mathcal{F}}
\newcommand{\LL}{\mathcal{L}}
\newcommand{\bvt}{\bar\vartheta}
\newcommand{\om}{\omega}
\newcommand{\sm}{\sigma}
\newcommand{\bsm}{\bar\sigma}
\newcommand{\ta}{\theta}
\newcommand{\lm}{\lambda}
\newcommand{\gm}{\gamma}
\newcommand{\dt}{\delta}
\newcommand{\al}{\alpha}
\newcommand{\bt}{\beta}
\newcommand{\bs}{(\bg,\A\bvt,\A\bsm)}

\DeclareMathOperator{\D}{d\!} 
\newcommand{\ddx}[1]{\dfrac{\partial}{\partial x^{#1}}}
\newcommand{\ddy}[1]{\dfrac{\partial}{\partial y^{#1}}}
\newcommand{\ddt}{\dfrac{\partial}{\partial t}}

\DeclareMathOperator{\grad}{grad} 
\DeclareMathOperator{\tr}{tr} 
\DeclareMathOperator{\Span}{span} 

\newcommand{\thmref}[1]{Theorem~\ref{#1}}

\newcommand{\corref}[1]{Corollary~\ref{#1}}

\newcommand{\A}{\allowbreak{}}

\begin{document}

\title[Almost Riemann Solitons with Vertical Potential on  ...]
{Almost Riemann Solitons with Vertical Potential on Conformal Cosymplectic Contact Complex Riemannian Manifolds}

\author{Mancho Manev}

\address[M. Manev]{
University of Plovdiv Paisii Hilendarski,
Faculty of Mathematics and Informatics,
Department of Algebra and Geometry,
24 Tzar Asen St.,
Plovdiv 4000, Bulgaria
\&
Medical University -- Plovdiv,
Faculty of Pharmacy,
Department of Medical Physics and Biophysics,
15A Vasil Aprilov Blvd,
Plovdiv 4002, Bulgaria}
\email{mmanev@uni-plovdiv.bg}

\dedicatory{Dedicated to the memory of the author's irreplaceable teacher, colleague and friend Prof. Kostadin Gribachev (1938--2022).}

\begin{abstract}
Almost Riemann solitons are introduced and studied on an almost contact complex Rie\-mannian manifold,
\ie an almost contact B-metric manifold, obtained from a cosymplectic manifold of the considered type by a contact conformal transformation of the Reeb vector field, its dual contact 1-form, the B-metric, and its associated B-metric.
The potential of the studied soliton is assumed to be in the vertical distribution,
\ie it is collinear to the Reeb vector field.
In this way, manifolds from the four main classes of the studied manifolds are obtained.
Curvature properties of the resulting manifolds are derived. An explicit example of dimension five is constructed.
The Bochner curvature tensor is used (for dimension at least seven) as a conformal invariant to get properties and construct an explicit example in relation to the obtained results.
\end{abstract}



\maketitle


\section{Introduction}

The Riemann flow was introduced by C.\ Udri\c{s}te in \cite{Udr10,Udr12} and it is the flow associated with the evolution
equation
\[
\frac{\partial}{\partial t}\left(g\owedge g\right)(t)=-4R (g(t)),
\]
where $R$ is the Riemann curvature tensor of type $(0, 4)$ corresponding to the metric $g$ at time $t$
and $\owedge$
stands for the Kulkarni--Nomizu product of two symmetric tensors of type $(0,2)$;
e.g.\ this product has the following form for order-2 covariant tensors $g$ and $h$
\begin{equation}\label{KNp}
\begin{array}{l}
\left(g\owedge h\right)(x,y,z,w)
=g(y,z) h(x,w) - g(x,z)h(y,w)\\[4pt]
\phantom{g\owedge h(x,y,z,w)=h}
+h(y,z) g(x,w) - h(x,z) g(y,w).
\end{array}
\end{equation}
Here and further $x$, $y$, $z$, $w$ stand for arbitrary vector fields on a smooth manifold $M$.

Riemann solitons are introduced by I.-E.\ Hiric\u{a} and C.\ Udri\c{s}te in \cite{HirUdr16}. They are critical metrics for Riemann flow as they are self-similar solutions of its evolution equation,
\ie it evolves over time from a given Riemannian metric on $M$ by diffeomorphisms and
dilatations.

A Riemannian metric $g$ on a smooth manifold $M$ is said to be a
\emph{Riemann soliton} if there exists a differentiable vector field $\vt$ and a real constant $\sm$ such that \cite{HirUdr16}
\begin{equation*}\label{RiemSol}
    2R + \sm g \owedge g + g \owedge \LL_\vt g = 0,
\end{equation*}
where $\LL_\vt$ is the Lie derivative along $\vt$.
Such a vector field $\vt$ is known as the \emph{potential of the soliton}.
In the case when $\sm$ is a differentiable function on $M$, then $g$ is called an \emph{almost Riemann soliton}.
If $\vt$ is Killing, \ie $\LL_\vt g=0$, then $M$ is a manifold of constant sectional curvature. In this sense, the Riemann soliton is
a generalization of space of constant curvature.
%

In the beginning \cite{HirUdr16}, the notion of Riemann soliton was studied in the
context of Sasakian geometry and it was named as Sasaki-Riemann soliton.


In recent years, some interesting results have been obtained for Riemann solitons and almost Riemann solitons on almost contact metric manifolds.
In \cite{VenKumDev20,DevKumVen21}, Venkatesha, Devaraja and Kumara study the cases of almost Kenmotsu
manifolds and K-contact manifolds.
Biswas, Chen and U. C. De characterize almost co-K\"ahler manifolds whose metrics are
Riemann solitons in \cite{BiChDe22}.
K. De and U. C. De prove in \cite{DeDe22} some geometric properties of almost Riemann solitons on non-cosymplectic normal almost contact metric manifolds and in particular on quasi-Sasakian 3-dimensional manifolds.
In \cite{ChiVen21}, Chidananda and Venkatesha study Riemann solitons
on non-Sasakian $(\kappa,\mu)$-contact manifolds in relation with the $\eta$-Einstein property, where the potential is an infinitesimal contact transformation or collinear to the Reeb vector field.

A.-M. Blaga contributes to the study of Riemann and almost Riemann soliton
with \cite{Bl20a} for Riemannian manifolds together with La\c{t}cu
and
with \cite{Bl21} for $(\al,\bt)$-contact metric manifolds.
In the latter case, compact Riemann solitons with constant-length potential are shown to be trivial.
This result was extended by  Tokura, Barboza, Batista, and Menezes
in \cite{ToBaBaMe22}  without additional conditions on the potential.

$\mathcal{D}$-homothetic deformations are introduced by S.\ Tanno \cite{Tanno68}
in almost contact metric geometry, where $\mathcal{D}$ denotes the contact distribution.
These transformations 
preserve the property of a structure to be K-contact or Sasakian.
In \cite{Bl22}, Blaga study almost Riemann solitons on a $\mathcal{D}$-homothetically deformed Kenmotsu manifold with different conditions on the potential and explicitly obtain Ricci and scalar curvatures for some cases.


An \emph{almost contact complex Riemannian} (or accR for short) manifold is an odd-dimension\-al pseudo-Riemannian manifold $M$ equipped with
a B-metric $g$ and
an almost contact structure $(\f,\xi,\eta)$ and therefore $M$ has a codimension-one distribution $\mathcal{H} = \ker \eta$ equipped with a complex Riemannian structure. These manifolds are also known as \emph{almost contact B-metric manifolds} \cite{GaMiGr}.

What mainly distinguishes an accR structure from the better-known almost contact metric structure is the presence of another metric of the same type associated with the given metric.
Both B-metrics have neutral signature on $\mathcal{H}$ and the restriction of $\f$ on $\mathcal{H}$  (actually, an almost complex structure) acts as an anti-isometry on the metric.
Manifolds of this type have been studied and investigated for example in
\cite{GaMiGr},
\cite{HM}--\cite{NakGri2}.

The aim of this paper is to investigate the interaction between almost Riemann solitons and the accR structure.
One way to realize this goal
is to use conformal transformations of the accR structure.
Contact conformal transformations of B-metric were introduced and initially studied in \cite{ManGri1} and \cite{ManGri2} by K. Gribachev and the author.
The metric deformation depends on both the two B-metrics and their restriction of the vertical distribution determined by $\xi$.
A generalization of these transformations and the $\mathcal{D}$-homothetic deformations of the accR structure (introduced in \cite{Bulut19}) that uses a triplet of functions on the manifold are the following transformations.
Contact conformal transformations of a general type that transform not only the B-metrics but also $\xi$ and $\eta$ are studied in \cite{Man4}.
According to this work, the class of accR manifolds,
which are closed under the action of these transformations, is the direct sum
of the four main classes among the eleven basic classes of these manifolds known from
the classification of Ganchev--Mihova--Gribachev in \cite{GaMiGr}.
Main classes are called those for
which their manifolds are characterized by the fact that the covariant derivative of the
structure tensors with respect to the Levi-Civita connection of any of the B-metrics is
expressed only by a pair of B-metrics and the corresponding traces.

The present paper is organized as follows. After the present introductory words,
Section 2 recalls the basic concepts of accR manifolds and contact
conformal transformations of the structure tensors on them.
Section 3 introduces the
notion of almost Riemann soliton with vertical potential on a transformed accR manifold and show condition that implies flatness of the  manifold.
Section 4 presents the curvature properties of contact conformal accR manifolds that are transformed from such manifolds of cosymplectic type and admit the studied soliton.
Section 5 is devoted to the particular case of the situation in the previous section when the transformed manifold is also of cosymplectic type.
The last two sections provide explicit examples for the studied manifolds in relation with the obtained results.

\section{Almost Contact Complex Riemannian Manifolds and Their Contact Conformal Transformations}

We study \emph{almost contact complex Riemannian manifolds} or \emph{accR manifolds} for short,
also known as \emph{almost contact B-metric manifolds}.
 Such a manifold, denoted by $\M$, is a $(2n+1)$-dimensional differentiable manifold, which is equipped with an almost
contact structure $(\f,\xi,\eta)$ and B-metric $g$. This means that $\f$ is an endomorphism
of the tangent bundle $TM$, $\xi$ is a Reeb vector field and $\eta$ is its dual contact 1-form.
Moreover,
$g$ is a pseu\-do-Rie\-mannian
metric of signature $(n+1,n)$ satisfying the following
algebraic relations: \cite{GaMiGr}
\begin{equation}\label{strM}
\begin{array}{c}
\f\xi = 0,\qquad \f^2 = -\I + \eta \otimes \xi,\qquad
\eta\circ\f=0,\qquad \eta(\xi)=1,\\
g(\f x, \f y) = - g(x,y) + \eta(x)\eta(y),
\end{array}
\end{equation}
where $\I
$ is the identity transformation on the set $\Gamma(TM)$ of vector fields on $M$.

As consequences of \eqref{strM} are known the following equations
\begin{equation*}\label{strM2}
\begin{array}{ll}
g(\f x, y) = g(x,\f y),\qquad 
g(x, \xi) = \eta(x),\qquad
g(\xi, \xi) = 1,\qquad 
\eta(\n_x \xi) = 0,
\end{array}
\end{equation*}
where $\n$ is the Levi-Civita connection of $g$.

The investigated manifold $\M$ has another B-metric besides $g$.
This is the associated metric $\g$ of $g$ on $M$ defined by
\[
\g(x,y)=g(x,\f y)+\eta(x)\eta(y).
\]
Obviously, $\g$ as well as $g$ satisfies also the last condition in \eqref{strM} and has the same signature.

A classification of accR manifolds containing eleven basic classes $\F_1$, $\F_2$, $\dots$, $\F_{11}$ is given in
\cite{GaMiGr}. This classification is made with respect
to the tensor $F$ of type $(0,3)$ defined by
\begin{equation*}\label{F=nfi}
F(x,y,z)=g\bigl( \left( \nabla_x \f \right)y,z\bigr).
\end{equation*}
The following identities are valid:
\begin{equation*}\label{F-prop}
\begin{array}{l}
F(x,y,z)=F(x,z,y)
=F(x,\f y,\f z)+\eta(y)F(x,\xi,z)
+\eta(z)F(x,y,\xi),\\[4pt]
F(x,\f y, \xi)=(\n_x\eta)y=g(\n_x\xi,y).
\end{array}
\end{equation*}

The special class $\F_0$,
determined by the condition $F=0$, is the intersection of the basic classes and it is known as the
class of the \emph{cosymplectic accR manifolds}.
Sometimes, in the context of classification and for brevity, these manifolds are called $\F_0$-manifolds.

Let $\left\{e_i;\xi\right\}$ $(i=1,2,\dots,2n)$ be a basis of
$T_pM$ and let $\left(g^{ij}\right)$ be the inverse matrix of the
matrix $\left(g_{ij}\right)$ of $g$. Then the following 1-forms
are associated with $F$:
\begin{equation*}\label{t}
\theta=g^{ij}F(e_i,e_j,\cdot),\quad
\theta^*=g^{ij}F(e_i,\f e_j,\cdot), \quad \omega=F(\xi,\xi,\cdot).
\end{equation*}
These 1-forms are known also as the \emph{Lee forms} of the considered manifold. Obviously, the identities
$\om(\xi)=0$ and $\ta^*\circ\f=-\ta\circ\f^2$ are always valid.


In \cite{ManGri1}, it is introduced the so-called contact conformal transformation of the B-metric $g$. It maps $g$ into a new B-metric $\bar g$ using both the B-metrics. 
Later, in \cite{Man4}, this transformation is generalized as a contact conformal transformation that gives an accR structure $(\f,\bar\xi,\bar\eta,\bar g)$ as follows
\begin{equation}\label{cct}
\begin{array}{l}
\bar\xi=e^{-w}\xi,\qquad \bar\eta=e^{w}\eta,\\[4pt]
\bar g= e^{2u}\cos{2v}\, g + e^{2u}\sin{2v}\, \g + \left(e^{2w}-e^{2u}\cos{2v}-e^{2u}\sin{2v}\right)\eta\otimes\eta,
\end{array}
\end{equation}
where $u, v, w$ are differentiable functions on $M$. The group of these transformations is denoted by $G$ and briefly we call each of the elements of $G$ a \emph{$G$-transformation}.

Let us note that the $G$-transformations of $(\f,\xi,\eta,g)$ are a generalization of the 
\emph{$\mathcal{D}$-homo\-thetic deformations}, where $\mathcal{D}$ denotes the contact distribution $\ker\eta$. Namely, for a positive constant $t$, a $\mathcal{D}$-homothetic deformation is defined by \cite{Bulut19} (see also \cite[p.~125]{Blair76} for the metric case)
\[
\bar \xi = \lm^{-1}\xi,\qquad
\bar \eta = \lm\,\eta,\qquad
\bg = -\lm\,g+\lm(\lm+1)\eta\otimes\eta.
\]
It is clear that $\mathcal{D}$-homothetic deformation is a $G$-transformation of the accR structure $(\f,\A\xi,\A\eta,\A g)$ for constants $u=\frac12 \ln \lm$, $v=\frac{\pi}{2}$ and $w=\ln \lm$.

The structure $(\f,\xi,\eta,g)$ determines two mutually orthogonal distributions with respect to $g$.
They are the horizontal (contact) distribution $\mathcal{H}=\ker\eta$ and the vertical distribution $\mathcal{V}=\Span\xi$. They coincide with the respective distributions for the structure $(\f,\bar\xi,\bar\eta,\bar g)$, \ie $\mathcal{H}=\bar{\mathcal{H}}=\ker\bar\eta$ and $\mathcal{V}=\bar{\mathcal{V}}=\Span\bar\xi$, due to the equalities in the first line of \eqref{cct}.

The corresponding tensors $F$ and $\bar F$ for the accR structures $(\f,\xi,\eta,g)$ and $(\f,\bar\xi,\bar\eta,\bar g)$ 
 by a $G$-transformation \eqref{cct} are related as follows (e.g.\ \cite{Man4}, see also \cite[(22)]{ManIv38})
\begin{equation}\label{ff}
\begin{aligned}
    2\bar{F}&(x,y,z)= 2e^{2u}\cos{2v}\, F(x,y,z)
    +e^{2u}\sin{2v} \left[P(x,y,z)+P(x,z,y)\right]\\[4pt]
    &+(e^{2w}-e^{2u}\cos{2v})\left[Q(x,y,z)+Q(x,z,y)+Q(y,z,x)+Q(z,y,x)\right]\\[4pt]
    &-2e^{2u}\left[\gm(z)g(\f x,\f y)+\dt(z)g(x,\f y)
    +\gm(y)g(\f x,\f z)+\dt(y)g(x,\f z)\right]
\\[4pt]
    &+2e^{2w}\eta(x)\left[\eta(y)\D w(\f z)+\eta(z)\D w(\f
    y)\right],
\end{aligned}
\end{equation}
where for  brevity we use the notations
\begin{equation*}\label{PQ}
\begin{array}{l}
P(x,y,z)=F(\f y,z,x)-F(y,\f z,x)+F(x,\f y,\xi)\eta(z), \\[4pt]
Q(x,y,z)=\left[F(x,y,\xi)+F(\f y,\f x,\xi)\right]\eta(z), \\[4pt]
\gm(z)= \cos{2v}\,\al(z)+\sin{2v}\,\bt(z),\qquad
\dt(z)= \cos{2v}\,\bt(z)-\sin{2v}\,\al(z),
\end{array}
\end{equation*}
\begin{equation}\label{albt}
\al = \D u\circ \f + \D v, \qquad \bt = \D u - \D v\circ \f.
\end{equation}


In the general case, the relations between the Lee forms of the corresponding manifolds $\M$ and $(M,\f,\bar\xi,\bar\eta,\bar g)$ are the following (see \cite{Man4})
\begin{equation}\label{ttbartt}
    \bar{\ta} = \ta+2n\, \al,\qquad
    \bar{\ta}^* = \ta^* +2n\,\bt,\qquad
    \bar{\om}  = \om +\D w\circ \f.
\end{equation}

As proved in \cite[Theorem 4.2, p.~62]{ManA}, the class of accR manifolds that is preserved by $G$-trans\-formations is the direct sum of all main classes $\F_1\oplus\F_4\oplus\F_5\oplus\F_{11}$, denoted here briefly as $G(\F_0)$.
The main classes are the only classes of accR manifolds in the Gan\-chev--Mihova--Gribachev classification, where $F$ is expressed only by the metric $(0,2)$-tensors $g$, $\g$, $\eta\otimes\eta$ and the Lee forms.
The class $G(\F_0)$ obviously  contains $\F_0$.

\section{Almost Riemann Solitons with Vertical Potential on Contact Conformal accR Manifolds}

\begin{definition}
One says that the B-metric $\bg$ generates a \emph{Riemann soliton} with potential $\bvt$ and constant $\bsm$, denoted $(\bg,\bvt,\bsm)$, on an accR manifold $(M,\f,\bar\xi,\bar\eta,\bg)$, if the following condition is satisfied
\begin{equation}\label{RS}
    2\bar R + \bsm\, \bg \owedge \bg + \bg \owedge \LL_{\bvt} \bg = 0,
\end{equation}
where $\bar R$ is the Riemannian curvature tensor of $(M,\f,\bar\xi,\bar\eta,\bg)$ for $\bg$.
If $\bsm$ is a differentiable function on $M$, then the generated soliton is called \emph{almost Riemann soliton} $(\bg,\bvt,\bsm)$ on $(M,\f,\bar\xi,\bar\eta,\bg)$.
\end{definition}

In this work, we consider the case when the potential $\bvt$ is a vertical vector field, \ie  $\bvt$ is collinear to $\bar\xi$. Then we have the expression $\bvt=\bar k\,\bar\xi$ for a differentiable function $\bar k$ on the manifold. Obviously, the equality $\bar k=\bar\eta(\bvt)$ holds.
We require that the potential $\bvt$ does not degenerate at any point on the manifold $(M,\f,\bar\xi,\bar\eta,\bg)$. This means that $\bar k$ does not vanish anywhere, \ie  $\bar k\neq 0$.

It is well known the following expression of the Lie derivative in terms of the covariant derivative with respect to the Levi-Civita connection $\bar\n$ of 
$\bg$
\begin{equation*}\label{L1}
  \left(\LL_{\bar\xi} \bg\right)(x,y) = \bg\left(\bar\n_x \bar\xi, y\right)+\bg\left(x, \bar\n_y \bar\xi\right).
\end{equation*}
Similarly, the following formula can be obtained
\begin{equation}\label{Lbvt}
    \left(\LL_{\bvt} \bg\right)(x,y)=\bg \left(\bar\n_x\bvt,y\right)+\bg \left(x,\bar\n_y\bvt\right).
\end{equation}

For a vertical potential we have
$
\bar\n_x\bvt=\bar\n\left(\bar k\,\bar\xi\right)=\bar k\bar\n\bar\xi + \D \bar k (x)\bar\xi.
$
Then, the latter two equalities imply the formula
\begin{equation}\label{LL}
    \left(\LL_{\bvt} \bg\right)(x,y) =  \bar k \left(\LL_{\bar\xi} \bg\right)(x,y)+\bar h_1(x,y),
\end{equation}
where we use the following notation
\begin{equation}\label{h1}
\bar h_1(x,y)=\D \bar k(x)\bar\eta(y) + \D \bar k(y)\bar\eta(x).
\end{equation}
Obviously, the $(0,2)$-tensor $\bar h_1$ is symmetric and has the properties
\[
\bar h_1(\f x,\f y)=0, \qquad \bar h_1(\bar\xi,\bar\xi)=\tr{\bar h_1} =2\D \bar k(\bar\xi).
\]
Therefore, it vanishes on $\mathcal{H}$.
Furthermore, $\bar h_1$ vanishes if and only if $\bar k$ is a constant.


The following theorem holds for an arbitrary $\F_0$-manifold $M$.
It is not necessary to assume that the structure of $M$ is obtained by some $G$-transformation.

\begin{theorem}\label{thm:aRs-F0}
Every $\F_0$-manifold 
admitting an almost Riemann soliton
with vertical potential is flat.
\end{theorem}
\begin{proof}
Let us consider an $\F_0$-manifold $\M$ admitting an almost Riemann soliton $(g,\A\vt,\A\sm)$ with vertical potential $\vt=k\,\xi$.
Then its curvature tensor for $g$ has the following form, similar to \eqref{RS}
\begin{equation*}\label{R-RS}
R = \frac12 \sm\, g \owedge g -\frac12 g \owedge \LL_{\vt} g.
\end{equation*}
Since $\n\xi$ vanishes on an $\F_0$-manifold, then we have $\n_x \vt=\n_x(k\xi)=\D{k}(x)\xi$.
Due to the equality
$\left(\LL_{\vt} g\right)(x,y)=g \left(\n_x\vt,y\right)+g \left(x,\n_y\vt\right)$,
which is the analogue of  \eqref{Lbvt} on $\M$, we get in this case that
$
    \LL_{\vt} g=h_1,
$ 
where we use the following notation similar to \eqref{h1}
\begin{equation}\label{h1-}
h_1(x,y)=\D k(x)\eta(y) + \D k(y)\eta(x).
\end{equation}

Thus, the curvature tensor of such a manifold $\M$ takes the form
\begin{equation}\label{R-RS=F0}
R = -\frac12 \sm\, g \owedge g -\frac12 g \owedge h_1.
\end{equation}

Using \eqref{h1-} 
and \eqref{R-RS=F0}, we get the Ricci tensor and scalar curvatures for $g$ and $\g$, respectively, given in the following expressions
\begin{gather}\label{ro-F0}
\rho=-\left\{2n\,\sm+\D k(\xi)\right\}g - \frac12(2n-1)h_1,
\\[4pt]
\label{tau-F0}
\tau=-2n\left\{(2n+1)\sm+2\D k(\xi)\right\},\qquad \tilde{\tau}=0.
\end{gather}

The Riemannian curvature tensor $R$ of an $\F_0$-manifold has the K\"ahler property
\begin{equation}\label{K-F0}
R(x,y,\f z,\f w)=-R(x,y,z,w)
\end{equation}
since $\f$, $\xi$ and $\eta$ are covariant constant on $M$ with respect to $\n$ \cite{ManGri2}.
As consequences of \eqref{K-F0} and \eqref{h1-} we have $\rho(\xi,\xi)=0$ and $h_1(\xi,\xi)=2\D{k}(\xi)$, respectively, which together with \eqref{ro-F0} imply
\begin{equation}\label{dksm}
\D{k}(\xi)=-\sm.
\end{equation}

On the other hand, by virtue of \eqref{R-RS=F0} and \eqref{KNp}, we obtain
\begin{equation}\label{Rff}
R(x,y,\f z,\f w)=-\frac12 \sm \left(g^* \owedge g^*\right) (x,y,z,w) -\frac12 \left(g^* \owedge h_1^*\right)(x,y,z,w),
\end{equation}
where we use the notations $g^*=g(\cdot,\f\cdot)$, $h_1^*=h_1(\cdot,\f\cdot)$.
Then, taking into account that $g^*$ and $h_1^*$ are traceless due to the properties of $\f$, equalities \eqref{K-F0} and \eqref{Rff} yield consequently
\begin{equation}\label{Rrt}
R=-\frac12 \sm\, g^* \owedge g^* +\frac12 g^* \owedge h_1^*,\qquad
\rho=\sm\left[g-\eta\otimes\eta\right],\qquad \tau=2n\sm,\qquad \tilde\tau=0.
\end{equation}
Comparing the values of $\tau$ in \eqref{Rrt} and \eqref{tau-F0}, we obtain $\D{k}(\xi)=-(n+1)\sm$, which due to \eqref{dksm} gives $\D{k}(\xi)=\sm=0$.
Therefore, from \eqref{Rrt} it follows $\rho=0$, which together with
\eqref{ro-F0} imply $h_1=0$ and then $R=0$, bearing in mind \eqref{R-RS=F0}.
\end{proof}


\section{
$G(\F_0)$-Manifolds Admitting
the Studied Solitons}

In this section, we consider $\M$ as an $\F_0$-manifold, \ie $F=0$.
Let the resulting accR manifold $(M,\f,\bar\xi,\bar\eta,\bg)$ by $G$-transformation be called
a \emph{$G(\F_0)$-manifold}.
Then, the following expression follows from \eqref{ff} and gives the form of the fundamental tensor of $(M,\A\f,\A\bar\xi,\A\bar\eta,\bg)$:
\begin{equation*}\label{ff0}
\begin{aligned}
    2\bar{F}(x,y,z)&=
    -2e^{2u}\bigl\{\gm(z)g(\f x,\f y)+\dt(z)g(x,\f y)\bigr.
    +\gm(y)g(\f x,\f z)+\dt(y)g(x,\f z)\bigr\}
\\[4pt]
    &\phantom{=\ }
    +2e^{2w}\eta(x)\left\{\eta(y)\D w(\f z)+\eta(z)\D w(\f y)\right\}.
\end{aligned}
\end{equation*}
Then, using \eqref{albt} and \eqref{ttbartt}, the corresponding Lee forms are specialized as follows
\begin{equation*}\label{ttbartt-F0}
    \bar{\ta} = 2n \left\{\D u\circ \f + \D v\right\},\qquad
    \bar{\ta}^* = 2n\left\{\D u - \D v\circ \f\right\},\qquad
    \bar{\om}  = \D w\circ \f.
\end{equation*}

\begin{theorem}\label{thm:R-GF0-aRs}
A $G(\F_0)$-manifold  $(M,\f,\bar\xi,\bar\eta,\bg)$ admitting an almost Riemann soliton $\bs$ 
with vertical potential has a curvature tensor of the following form
\begin{equation}\label{bR}
\begin{array}{l}
  \bar R= -\left[\frac{\bar \sm}{2} +\bar k\D{u}(\bar\xi)\right]\bg\owedge \bg
    -\bar k\left\{\D{v}(\bar\xi)\,\bg\owedge\tilde\bg\right.
    \left.
    -\left[\D{u}(\bar\xi)+\D{v}(\bar\xi)\right]\bg\owedge\left(\bar\eta\otimes\bar\eta\right)\right\}\\[4pt]
    \phantom{\bar R=}
-\frac12\, \bg\owedge\bar h_1-\frac12\, \bar k\,\bg\owedge\bar h_2,
\end{array}
\end{equation}
where
\begin{equation}\label{h2}
\bar h_2 (x,y)=\bar\eta(x)\D w(\f^2 y)+\bar\eta(y)\D w(\f^2 x).
\end{equation}
\end{theorem}
\begin{proof}
Bearing in mind \eqref{RS}, we have to determine $\LL_{\bvt} \bg$.
The expression of the Lie derivative of $\bg$ along $\bar\xi$ for a $G(\F_0)$-manifold is given in \cite{Man73} in the form:
\begin{equation}\label{Lxi0}
\begin{aligned}
    \left(\LL_{\bar\xi} \bg\right)(x,y) =& -2e^{2u-w}\left[\cos{2v}\,\D{u}(\xi)-\sin{2v}\,\D{v}(\xi)\right]g(\f x,\f y)
\\[4pt]
    & +2e^{2u-w}\left[\cos{2v}\,\D{v}(\xi)+\sin{2v}\,\D{u}(\xi)\right]g( x,\f y)
\\[4pt]
    &+e^{w}\left[\eta(x)\D w(\f^2 y)+\eta(y)\D w(\f^2 x)\right].
\end{aligned}
\end{equation}

Using the second line in \eqref{cct}, we derive the following formulas
\begin{equation*}\label{ggbargg}
\begin{array}{l}
    g(\f x,\f y) = e^{-2u}\left[\cos{2v}\,\bg(\f x,\f y)+\sin{2v}\,\bg(x,\f y)\right],
\\[4pt]
    g(x,\f y) = e^{-2u}\left[\cos{2v}\,\bg(x,\f y)-\sin{2v}\,\bg(\f x,\f y)\right],
\end{array}
\end{equation*}
which we apply in \eqref{Lxi0} together with the first line in \eqref{cct}. In this way we get
\begin{equation}\label{Lxi0=}
\begin{aligned}
    \left(\LL_{\bar\xi} \bg\right)(x,y) =& -2\left[\D{u}(\bar\xi) \bg(\f x,\f y) -\D{v}(\bar\xi)\bg(x,\f y)\right]
+\bar h_2 (x,y),
\end{aligned}
\end{equation}
where we introduce the notation \eqref{h2}.
Obviously, $\bar h_2$ is a symmetric $(0,2)$-tensor having the following properties
\[
\bar h_2(\f x,\f y)=\bar h_2(\bar\xi,\bar\xi)=\tr{\bar h_2}=0,\qquad
\bar h_2(x,y)=\bar h_2(x,\bar\xi)\bar\eta(y)+\bar\eta(x)\bar h_2(\bar\xi,y).
\]
Moreover, the formula $\bar h_2(x,\bar\xi)=\D{w}(\f^2x)$ is valid. It is easy to conclude that
$\bar h_2$ vanishes if and only if the function $w$ is constant on $\mathcal{H}$, \ie  $\D w\circ\f^2=0$.

The formula in \eqref{Lxi0=} can be rewritten in the form
\begin{equation*}\label{Lxi0==}
    \LL_{\bar\xi} \bg =
    2\left\{\D{u}(\bar\xi) \bg +\D{v}(\bar\xi)\tilde\bg-\left[\D{u}(\bar\xi)+\D{v}(\bar\xi)\right]\bar\eta\otimes\bar\eta\right\}
+\bar h_2.
\end{equation*}
Then we substitute the last equality in \eqref{LL} and get the following
\begin{equation*}\label{LL=}
    \LL_{\bvt} \bg =   2\bar k\left\{\D{u}(\bar\xi) \bg +\D{v}(\bar\xi)\tilde\bg-\left[\D{u}(\bar\xi)+\D{v}(\bar\xi)\right]\bar\eta\otimes\bar\eta\right\}
+\bar h_1+\bar k\bar h_2.
\end{equation*}

Using the Kulkarni-Nomizu product for $\bg$ and the last obtained Lie derivative, we obtain
\begin{equation}\label{KN-LL}
\begin{array}{l}
    \bg\owedge\LL_{\bvt} \bg =   2\bar k\left\{\D{u}(\bar\xi)\,\bg\owedge \bg
    +\D{v}(\bar\xi)\,\bg\owedge\tilde\bg
    -\left[\D{u}(\bar\xi)+\D{v}(\bar\xi)\right]\bg\owedge\left(\bar\eta\otimes\bar\eta\right)\right\}\\[4pt]
    \phantom{\bg\owedge\LL_{\bvt} \bg =}
+\bg\owedge\bar h_1+\bar k\,\bg\owedge\bar h_2.
\end{array}
\end{equation}
Then, according to \eqref{RS} and \eqref{KN-LL}, we establish the truthfulness of the statement.
\end{proof}

Taking the trace of \eqref{bR}, we obtain the expression of the Ricci tensor of the almost Riemann soliton satisfying the conditions of \thmref{thm:R-GF0-aRs} as follows
\begin{equation}\label{bro}
\begin{array}{l}
  \bar \rho = -\left[2n \,\bar \sm+\D\bar k(\bar\xi)+(4n-1) \bar k\D{u}(\bar\xi)\right]\bg\\[4pt]
    \phantom{\bar \rho =}
    -(2n-1)\bar k\left\{\D{v}(\bar\xi)\,\tilde\bg
    -\left[\D{u}(\bar\xi)+\D{v}(\bar\xi)\right]\,\bar\eta\otimes\bar\eta\right\}
    -\frac12(2n-1)\left[ \bar h_1+\bar k\,\bar h_2\right].
\end{array}
\end{equation}
Now, we take the trace of the Ricci tensor in \eqref{bro} to obtain the scalar curvature of $(M,\f,\A\bar\xi,\A\bar\eta,\A\bg)$ 
as follows
\begin{equation}\label{btau}
\begin{array}{l}
  \bar \tau = -2n\left[(2n+1)\bar \sm +2\D\bar k(\bar\xi)+4n \bar k \D{u}(\bar\xi)\right].
\end{array}
\end{equation}

After that we compute the associated quantity $\bar\tau^*$ of $\bar\tau$ defined by $\bar\tau^*=\bg^{ij}\f^k_j \bar\rho_{ik}$ and using \eqref{bro}, we obtain
\begin{equation}\label{btau*}
\begin{array}{l}
  \bar \tau^* = 2n(2n-1) \bar k \D{v}(\bar\xi).
\end{array}
\end{equation}

For every $\F_0$-manifold, the relation
 $\tilde{\tau}=-\tau^*$ is known from \cite{Man3}, where
 $\tilde{\tau}$ is the scalar curvature for $\tilde{g}$.
 But for $G(\F_0)$-manifolds, which are outside of $\F_0$, this is not true, so there we use the so-called
 \emph{$*$-scalar curvature}.

\begin{corollary}\label{cor:v}
A $G(\F_0)$-manifold $(M,\f,\bar\xi,\bar\eta,\bar g)$ with an almost Riemann soliton $(\bg,\bvt,\bsm)$ and a vertical potential  has vanishing $*$-scalar curvature if and only if the function $v$ is a vertical constant, \ie  $\D{v}(\bar\xi)=0$.
\end{corollary}
\begin{proof}
The statement follows from \eqref{btau*} and the condition that $\bar k$ is not identically zero,
otherwise it would lead to a degeneration of the potential $\bvt$.
\end{proof}


In \cite{Man62}, the notion of
an \emph{Einstein-like} accR manifold $\M$ is introduced by the following condition for its Ricci tensor
\begin{equation}\label{defEl}
\begin{array}{l}
\rho=a\,g +b\,\g +c\,\eta\otimes\eta,
\end{array}
\end{equation}
where $(a,b,c)$ is some triplet of constants.
In particular, when $b=0$ and $b=c=0$, the manifold is called an \emph{$\eta$-Einstein manifold} and an \emph{Einstein manifold}, respectively.
If $a$, $b$, $c$ are functions on $M$, then the manifold satisfying condition \eqref{defEl} is called \emph{almost Einstein-like}, and in particular for $b=0$ and $b=c=0$---\emph{almost $\eta$-Einstein} and \emph{almost Einstein}, respectively.

\begin{theorem}\label{thm:aEl}
A $G(\F_0)$-manifold $(M,\f,\bar\xi,\bar\eta,\bar g)$ with an almost Riemann soliton $(\bg,\bvt,\bsm)$ and a vertical potential is an almost Einstein-like manifold if and only if the condition 
$\bar k=\lm e^w$ ($\lm=const$) is satisfied on $\mathcal{H}$.

The almost Einstein-like manifold $(M,\f,\bar\xi,\bar\eta,\bar g)$ has the following Ricci tensor
\begin{gather}\label{bro-aEl}
\begin{array}{l}
  \bar \rho = -\left[2n \,\bar \sm+\D\bar k(\bar\xi)+(4n-1) \bar k\D{u}(\bar\xi)\right]\bg
                -(2n-1)\bar k\D{v}(\bar\xi)\,\tilde\bg
  \\[4pt]
    \phantom{\bar \rho =}
    -(2n-1)\left\{\D\bar k(\bar\xi)-\bar k\left[\D{u}(\bar\xi)+\D{v}(\bar\xi)\right]
    \right\}\bar\eta\otimes\bar\eta,\\[4pt]
\end{array}
\end{gather}
and the expressions of the scalar curvatures are the same as in \eqref{btau} and \eqref{btau*}.
\end{theorem}
\begin{proof}
Bearing in mind \eqref{bro} and the definition of almost Einstein-like accR manifold, we conclude that the considered manifold is almost Einstein-like if and only if $\bar h_1+\bar k\,\bar h_2$ is a function multiple of $\bar\eta\otimes\bar\eta$. Therefore, we have the following condition due to \eqref{h1} and \eqref{h2}
\begin{equation}\label{dkdw}
\D \bar k+ \bar k\D w\circ \f^2 =f\,\bar\eta,
\end{equation}
where $f$ is an arbitrary function on the manifold.
%
An immediate consequence of \eqref{dkdw} for $\bar\xi$ is $f=\D \bar k(\bar \xi)$. Then, we obtain
\begin{equation}\label{h1h2}
\bar h_1+\bar k\,\bar h_2=2\D \bar k(\bar \xi)\bar\eta\otimes\bar\eta,
\end{equation}
and in particular for $f=0$ the function $\bar k$ is a vertical constant.

Applying $\f$ to the argument of \eqref{dkdw}, we get the following consequence
\[
\D \bar k\circ\f - \bar k\D w\circ\f =0.
\]
Since $\bar k$ is not zero, the last equation has the following solution
$\bar k =\lm e^w$ restricted on $\mathcal{H}$, where $\lm$ is an arbitrary constant.

The formula in \eqref{bro-aEl} follows from \eqref{bro} and \eqref{h1h2}. It implies the same expressions of $\bar\tau$ and $\bar\tau^*$ as in \eqref{btau} and \eqref{btau*}.
\end{proof}

\begin{theorem}\label{thm:aEl-eEl-El}
Let $(M,\f,\bar\xi,\bar\eta,\bar g)$ be a $G(\F_0)$-manifold  with an almost Riemann soliton $(\bg,\bvt,\bsm)$ and a vertical potential. Then $(M,\f,\bar\xi,\bar\eta,\bar g)$ is:
\begin{enumerate}
  \item[(i)] an almost $\eta$-Einstein manifold if and only if $v$ is a vertical constant, \ie $\D{v}(\bar\xi)=0$;
  \item[(ii)] an almost Einstein manifold if and only if $v$ is a vertical constant and the condition
        $\bar k=\mu e^{u}$ ($\mu=const$) is satisfied on $\mathcal{V}$.
\end{enumerate}
\end{theorem}
\begin{proof}
The statements in (i) and (ii) are easily derived by considering the particular cases of \eqref{defEl} that are reflected in \eqref{bro-aEl}. The equality in (ii) is a solution of $\D{\bar k}(\bar\xi)=\bar k\D{u}(\bar\xi)$.
\end{proof}

\begin{corollary}\label{cor:aEl-eEl-El}
Let $(M,\f,\bar\xi,\bar\eta,\bar g)$ be a $G(\F_0)$-manifold  with an almost Riemann soliton $(\bg,\bvt,\bsm)$ and a vertical potential. Then the Ricci tensor and the scalar curvatures are the following when $(M,\f,\bar\xi,\bar\eta,\bar g)$ is:
\begin{enumerate}
  \item[(i)] an almost $\eta$-Einstein manifold:
\begin{equation*}\label{bro-aeE}
    \begin{array}{l}
  \bar \rho = -\left[2n \,\bar \sm+\D\bar k(\bar\xi)+(4n-1) \bar k\D{u}(\bar\xi)\right]\bg
    -(2n-1)\left[\D\bar k(\bar\xi)-\bar k\D{u}(\bar\xi)
    \right]\bar\eta\otimes\bar\eta,
\end{array}
\end{equation*}
$\bar\tau$ as in \eqref{btau} and $\bar{\tau}^*=0$.
  \item[(ii)] an almost Einstein manifold:
\begin{equation*}\label{bro-aE}
    \begin{array}{l}
  \bar \rho = -2n\left[\bar \sm+2\bar k\D{u}(\bar\xi)\right]\bg,\qquad
  \bar\tau=-2n (2n+1)\left[\bar \sm +2\bar k\D{u}(\bar\xi)\right],\qquad \bar\tau^*=0.
\end{array}
\end{equation*}
\end{enumerate}
\end{corollary}

\begin{corollary}\label{cor:aRs=El}
A $G(\F_0)$-manifold $(M,\f,\bar\xi,\bar\eta,\bar g)$ with an almost Riemann soliton $(\bg,\bvt,\bsm)$ and a vertical potential is an Einstein-like manifold if and only if the functions $\bar \sm$, $\bar k$, $u$ and $v$ satisfy the following conditions
\begin{equation}\label{usl1}
  \bar \sm +2\bar k\D{u}(\bar\xi)=const,\qquad
  \D{\bar k}(\bar\xi)-\bar k\D{u}(\bar\xi)=const,\qquad
                      \bar k\D{v}(\bar\xi)=const.
\end{equation}
Moreover, $(\bg,\bvt,\bsm)$ is a Riemann soliton with a vertical potential on the $\eta$-Einstein manifold if and only if
\end{corollary}
\begin{proof}
The considered manifold is Einstein-like if and only if the three coefficients of $(0,2)$-tensors in \eqref{bro-aEl} are constants. This system of equations is equivalent to the equations in \eqref{usl1}.
The case for the Riemann soliton follows from $\bar\sm=const$ and \eqref{usl1}.
\end{proof}


\subsection{Example of an $\F_5$-manifold of dimension 5}\label{ex-4.1}

A trivial example of an $\F_0$-manifold $\left(\R^{2n+1}, \f, \xi, \eta, g\right)$ of arbitrary dimension is given in \cite{GaMiGr}.
An accR structure is defined in the space
$\mathbb{R}^{2n+1}=\left\{\left(x^1,\dots,x^n;y^1,\dots,y^n;t\right)\right\}$ 
by the following way
\begin{equation*}\label{ex-str}
\begin{array}{c}
\f\ddx{i}=\ddy{i},\qquad \f\ddy{i}=-\ddx{i},\qquad \f\ddt=0,\qquad \xi=\ddt, \qquad \eta=\D{t},\\[4pt]
g(z,z)=-\delta_{ij} \lm^i\lm^j+\delta_{ij} \mu^i\mu^j+\nu^2,
\end{array}
\end{equation*}
where $z=\lm^i\ddx{i}+\mu^i\ddy{i}+\nu\ddt$ and $\delta_{ij}$ is the Kronecker delta.

In \cite{Man6} (see also \cite[Example 5, p.~105]{ManA}, we give an example of a pair of functions $(u,v)$ on $\R^{2n+1}$ as for dimension 5, \ie $n=2$, it can be written as follows
\begin{equation}\label{ex2-uv}
u=\ln\sqrt{\frac{|t|}{\left(x^1+y^2\right)^2+\left(x^2-y^1\right)^2}},\qquad
v=\arctan \frac{x^1+y^2}{x^2-y^1},
\end{equation}
where $x^1+y^2\neq 0$, $x^2-y^1\neq 0$ and $t\neq 0$.

It is shown that $(u,v)$ satisfy the condition $\D{v}=-\D{u}\circ\f$, because of which the $G$-trans\-formation determined by $(u,v,w=0)$ deforms the given $\F_0$-man\-i\-fold into an $\F_5$-manifold $\left(\R^{2n+1}, \f, \xi, \eta, \bg\right)$ defined by
$
\bar F(x,y,z)=-\frac{1}{4}\bar\ta^*(\xi)\left\{\bg(x,\f y)\eta(z)+\bg(x,\f z)\eta(y)\right\},$
where
$\bar\ta^*(\xi)=4\D{u}(\xi)=\frac{2}{t}$. Its curvature tensor, the scalar curvature
and the $*$-scalar curvature are given in the form
\begin{equation}\label{Rtau}
\bar R= \frac{\bar\tau}{32}\left\{\bg\owedge\bg-\bg\owedge\left(\eta\otimes\eta\right)\right\},\qquad
\bar\tau=-\frac{8}{t^2},\qquad \bar\tau^*=0.
\end{equation}

Now, let us introduce an almost Riemann soliton $(\bg,\bvt,\bsm)$ with vertical potential $\bvt=\bar k\eta$ on $\left(\R^{2n+1}, \f, \xi, \eta, \bg\right)$ assuming that we have the following functions
\begin{equation}\label{smk}
\bsm =\frac{1}{3t^2},\qquad \bar k=-\frac{1}{6t},
\end{equation}
which determine the soliton.

Bearing in mind \thmref{thm:R-GF0-aRs}, we check the expression of $\bar R$ in \eqref{bR}.
Using \eqref{ex2-uv} and \eqref{smk}, we compute successively $\bar h_2=0$ due to $w=0$, $\bar h_1=-\frac{1}{3t^2}\eta\otimes\eta$,
$\D{u}(\xi)=\frac{1}{2t}$, $\D{v}(\xi)=0$, and obtain  for the coefficients in \eqref{bR} the following
\begin{equation*}\label{koef-bR-exF5}
\begin{array}{l}
-\left[\frac{\bar \sm}{2} +\bar k\D{u}(\bar\xi)\right]=-\frac{1}{4t^2},\qquad
 -\bar k\D{v}(\bar\xi)=0,\qquad
 \bar k\left[\D{u}(\bar\xi)+\D{v}(\bar\xi)\right]=\frac{1}{12t^2}.
\end{array}
\end{equation*}
Then \eqref{bR} takes the form
\begin{equation}\label{bR-exF5}
  \bar R= -\frac{1}{4t^2}\left\{\bg\owedge \bg
-\bg\owedge\left(\eta\otimes\eta\right)\right\},
\end{equation}
which agree with \eqref{Rtau}. Thus, we verify \thmref{thm:R-GF0-aRs} and \corref{cor:v}.

Using \eqref{bR-exF5}, we get the following consequences
\[
\bar\rho = -\frac{1}{4t^2}\left\{7\bg -3\eta\otimes\eta\right\},\qquad
\bar\tau = -\frac{8}{t^2},\qquad \bar\tau^*=0.
\]
We thus conclude that the constructed manifold has negative scalar curvature and zero $*$-scalar curvature, and that it is almost $\eta$-Einstein (a particular case of almost Einstein-like), which is not almost Einstein.
These results support \corref{cor:v}, \thmref{thm:aEl},
\thmref{thm:aEl-eEl-El}(i) and \corref{cor:aEl-eEl-El}(i).

In addition, we can calculate the scalar curvature $\tilde{\bar\tau}$ with respect to $\tbg$.
In \cite{Man3} (see also \cite[Corollary 2.4, p.~38]{ManA}),
the relation of this quantity to the $*$-scalar curvature is expressed.
For the case under consideration, the given formula can be read in the following way
$
\tilde{\bar\tau} = -\bar\tau^* -\frac{5}{4}\left(\ta^*(\xi)\right)^2 - 2 \xi\left(\ta^*(\xi)\right).
$
Using that $\ta^*(\xi)=\frac{2}{t}$ in the present example, we get $\tilde{\bar\tau} = -\frac{1}{t^2}$, \ie it is also negative as $\bar\tau$.


\section{$\F_0$-Manifolds that are $G(\F_0)$-Manifolds and Admit the Studied Solitons
}

In this section, we consider an $\F_0$-manifold $\M$, \ie  $F=0$. Moreover, the resulting manifold   $(M,\f,\bar\xi,\bar\eta,\bar g)$ by a $G$-transformation is again in $\F_0$, \ie  $\bar F=0$.

To ensure that both considered manifolds are in $\F_0$, the transformation between them must be of a subgroup  $G_0$ of the group $G$ and defined by the following conditions \cite{Man4}
\begin{equation}\label{G0}
  \D{u}\circ \f = \D{v}\circ \f^2,\qquad  \D{u}(\xi)=\D{v}(\xi)=\D{w}\circ \f =0.
\end{equation}

    In this case, the relationship between the curvature tensors $R$ and $\bar R$ for $g$ and $\bar g$, respectively, is known from \cite[p.~83]{ManA} (see also \cite{Man4} and \cite{ManGri2}) and can be written as follows
\begin{equation}\label{RbarR-F0}
  \bar R = R -g\owedge S + g^* \owedge S^* + \left(\eta\otimes \eta\right)\owedge S,
\end{equation}
where $S^*=S(\cdot,\f\cdot)$ and
\begin{equation}\label{S-F0}
\begin{array}{l}
S = \n \D{u} +\D{u}\otimes\D{u}+\left(\D{u}\circ\f\right) \otimes\left(\D{u}\circ\f\right)
    +\frac12 \D{u}(\grad{u})[g-\eta\otimes \eta]
\\[4pt]
\phantom{S = \n \D{u} +\D{u}\otimes\D{u}+\left(\D{u}\circ\f\right) \otimes\left(\D{u}\circ\f\right)}
    -\frac12 \D{u}(\f\grad{u})[\g-\eta\otimes \eta].
\end{array}
\end{equation}

\begin{theorem}\label{thm:F0F0-tau}
Let $\M$ and its image via a $G_0$-transformation $(M,\f,\bar\xi,\bar\eta,\bar g)$ be $\F_0$-manifolds.
Then the corresponding scalar curvatures for the pair of B-metrics satisfy the relations
\begin{equation}\label{btt*-G0}
\begin{array}{l}
  \bar\tau = e^{-4u}\cos{4v} \left\{\tau - 4(n-1)\tr{S}\right\}
             +e^{-4u}\sin{4v} \left\{\tilde{\tau} + 4(n-1)\tr{S^*}\right\},\\[4pt]
  \tilde{\bar{\tau}} = e^{-4u}\cos{4v} \left\{\tilde{\tau} + 4(n-1)\tr{S^*}\right\}
             -e^{-4u}\sin{4v} \left\{\tau - 4(n-1)\tr{S}\right\},
\end{array}
\end{equation}
where
\begin{equation}\label{trStrS*}
\begin{array}{l}
  \tr S = \delta(\D{u})+2n\D{u}(\grad{u}),\qquad
  \tr S^* = \tilde\delta(\D{u})+2n\D{u}(\f\grad{u})
\end{array}
\end{equation}
for $\delta(\D{u})= g^{ij}\left(\n\D{u}\right)_{ij}$ and $\tilde\delta(\D{u})= \g^{ij}\left(\n\D{u}\right)_{ij}$.
\end{theorem}
\begin{proof}
Using \eqref{RbarR-F0} with \eqref{S-F0} for the corresponding curvature tensors $R$ and $\bar R$, by lengthy but standard calculations, we obtain expressions for the corresponding scalar curvatures given in \eqref{btt*-G0}.
\end{proof}

Obviously, the involved in \eqref{trStrS*} trace $\delta(\D{u})$ is actually the Laplacian of $u$ for $g$, usually denoted by $\Delta{u}$ or $\n^2{u}$, while $\tilde\delta(\D{u})$ is some kind of associated quantity of $\Delta u$ using $\g$.

%

\begin{corollary}\label{cor:aRsF0-F0}
Let $(\g,\bvt,\bsm)$ be an almost Riemann soliton with vertical potential on $(M,\f,\bar\xi,\bar\eta,\bar g)$ and let the requirements of \thmref{thm:F0F0-tau} be fulfilled.
Then $\M$ has constant scalar curvatures for both B-metrics $g$ and $\g$.
\end{corollary}
\begin{proof}
According to \thmref{thm:aRs-F0},  $(M,\f,\bar\xi,\bar\eta,\bar g)$ is flat, \ie $\bar R=0$ and therefore we have $\bar{\tau}=\tilde{\bar{\tau}}=\bsm=\D{\bar k}=0$.
Substituting the last equalities into \eqref{btt*-G0} and considering \eqref{trStrS*}, we get
\begin{equation}\label{tau-F0F0RS}
\begin{array}{l}
  \tau = 4(n-1)\Bigl\{\delta(\D{u})+2n\D{u}(\grad{u})\Bigr\},\\[4pt]
  \tilde{\tau} =
  -4(n-1)\left\{\tilde\delta(\D{u})+2n\D{u}(\f\grad{u})\right\}.
\end{array}
\end{equation}

In this way, we obtain the conditions that the scalar curvatures of an $\F_0$-manifold must satisfy in order to be mapped by a $G_0$-transformation to an $\F_0$-manifold admitting an almost Riemann soliton under study.

As a consequence of Theorem 5.2 in \cite[p.~81]{ManA}, we deduce for an $\F_0$-manifold that
the functions
$\arctan\frac{\tilde{\tau}}{\tau}$ and 
$\ln\sqrt{\tau^2+\tilde{\tau}^2}$ are constants, which implies that
$\tau$ and $\tilde{\tau}$ are constants.
\end{proof}

As is well known, the Bochner curvature tensor $B$ on a K\"ahler manifold can be considered in some sense as an analogue of the Weyl curvature tensor, and the vanishing of $B$ has remarkable geometric interpretations.

In \cite{ManGri2}, the \emph{Bochner curvature tensor of $\f$-holomorphic type} for a curvature tensor with K\"ahler property is introduced on an arbitrary accR manifold of dimension at least 7, \ie $n\geq 3$.  The Riemannian curvature tensor $R$ of an $\F_0$-manifold has the K\"ahler property
\eqref{K-F0} and the definition of the Bochner curvature tensor $B(R)$ as a tensor of type $(0,4)$ corresponding to $R$ can be written in the following form
\begin{equation}\label{BR}
\begin{array}{l}
  B(R)=R-\frac{1}{2(n-2)}\left\{g\owedge \rho -g^*\owedge \rho^*-(\eta\otimes\eta)\owedge \rho  \right\}\\[4pt]
  \phantom{B(R)=R}
  +\frac{1}{8(n-1)(n-2)}\left\{\tau \left[g\owedge g -g^*\owedge g^*-2(\eta\otimes\eta)\owedge g\right]\right.\\[4pt]
  \phantom{B(R)=R+\frac{1}{8(n-1)(n-2)}\left\{\right.}\left.
                                +2\tilde{\tau} \left[g\owedge g^*-(\eta\otimes\eta)\owedge g^*\right] \right\},
\end{array}
\end{equation}
where $\rho^*=\rho(\cdot,\f \cdot)$.

\begin{corollary}\label{cor:aRsF0-F0-LS}
Let $(\g,\bvt,\bsm)$ be an almost Riemann soliton with vertical potential on $(M,\f,\bar\xi,\bar\eta,\bar g)$ of dimension at least 7
and let the requirements of \thmref{thm:F0F0-tau} be fulfilled.
Then the Ricci tensor of $\M$ has the following form
\begin{equation}\label{SQ-F0B=0}
\begin{array}{l}
\rho = -2(n-2) S
+\frac{1}{4(n-1)}\left\{\tau\left[ g-  \eta\otimes \eta\right]
+\tilde{{\tau}}\left[\tilde{{g}}-  \eta\otimes \eta\right]\right\},
\end{array}
\end{equation}
where $S$ is determined by \eqref{S-F0}. Moreover, $\tr S$, $\tr S^*$, $\tau$ and $\tilde{{\tau}}$ are constants.
\end{corollary}
\begin{proof}
It is known from \cite{Man4}
that $B(R)$ on an $\F_0$-manifold is a contact conformal invariant of the group $G_0$, \ie $B(\bar R)=B(R)$.

For the considered manifold $(M,\f,\bar\xi,\bar\eta,\bar g)$ we obtained $\bar R=0$.
Hence $B(\bar R)$ also vanishes and this means
$B(R)=0$ for $\M$. Then, due to \eqref{BR}, we obtain an expression of $R$ as follows
\begin{equation}\label{R-BR=0}
\begin{array}{l}
  R=\frac{1}{2(n-2)}\left\{g\owedge \rho -g^*\owedge \rho^*-(\eta\otimes\eta)\owedge \rho  \right\}\\[4pt]
  \phantom{R=}
  -\frac{1}{8(n-1)(n-2)}\left\{\tau \left[g\owedge g -g^*\owedge g^*-2(\eta\otimes\eta)\owedge g\right]\right.
\left.
                                +2\tilde{\tau} \left[g\owedge g^*-(\eta\otimes\eta)\owedge g^*\right] \right\}
\end{array}
\end{equation}
where $\tau$ and $\tilde{\tau}$ are constants, according to \corref{cor:aRsF0-F0} and have values given in \eqref{tau-F0F0RS}.

Using that the Ricci tensor is hybrid with respect to $\f$, \ie $\rho=-\rho(\f\cdot,\f\cdot)$, on an $\F_0$-manifold, we can rewrite \eqref{R-BR=0} in the following more compact form
\begin{equation}\label{L-RbarR-F0}
   R = - g\owedge  L +  g^* \owedge  L^* + \left( \eta\otimes \eta\right)\owedge  L,
\end{equation}
where $ L$ is defined by
\begin{equation}\label{L}
 L=\frac{1}{2(n-2)}\rho
-\frac{1}{8(n-1)(n-2)}\left\{\tau\left[ g-  \eta\otimes \eta\right]
+\tilde{{\tau}}\left[\tilde{{g}}-  \eta\otimes \eta\right]
\right\}.
\end{equation}

The vanishing of $\bar R$ and \eqref{RbarR-F0} imply the following
\begin{equation}\label{Rbar0R-F0}
  R = g\owedge S - g^* \owedge S^* - \left(\eta\otimes \eta\right)\owedge S.
\end{equation}
Comparing \eqref{Rbar0R-F0} with \eqref{L-RbarR-F0}, we deduce that $S=-L$ and consequently \eqref{SQ-F0B=0} holds.
Equalities \eqref{trStrS*} and \eqref{tau-F0F0RS} imply that  $\tr S$ and $\tr S^*$ are also constants like $\tau$ and $\tilde{{\tau}}$.
\end{proof}


\subsection{Example of an $\F_0$-manifold of arbitrary dimension}

Let us consider again the $\F_0$-manifold $\left(\R^{2n+1}, \A\f, \A\xi, \A\eta, g\right)$ that was described at the beginning of \S\ref{ex-4.1}.

In \cite{ManGri1}, the following example of a pair of functions $(u,v)$ on an accR manifold is given
\begin{equation}\label{ex-uv}
u=\sum_{i=1}^n{\ln\sqrt{\left(x^i\right)^2+\left(y^i\right)^2}},\qquad
v=\sum_{i=1}^n \arctan \frac{y^i}{x^i}.
\end{equation}
It is shown that $(u,v)$ is a $\f$-holomorphic pair of functions, \ie the conditions for them in \eqref{G0} are satisfied.

Let $w$ be the function $e^t$. Then, we have $\D{w}=w\eta$, which implies $\D{w}\circ\f=0$. As a result, \eqref{G0} holds and $(u,v,w)$ determine a contact conformal transformation from $G_0$. This transformation deforms $\left(\R^{2n+1}, \f, \xi, \eta, g\right)$ into 
$(\R^{2n+1}, \f, \bar\xi, \bar\eta, \bg)$, which is again an $\F_0$-man\-i\-fold.

Bearing in mind \eqref{RbarR-F0} and the fact that $\left(\R^{2n+1}, \f, \xi, \eta, g\right)$ is flat, we obtain the curvature tensor of the resulting manifold in the form
\begin{equation*}\label{ex-RbarR-F0}
  \bar R = -g\owedge S + g^* \owedge S^* + \left(\eta\otimes \eta\right)\owedge S,
\end{equation*}
where $S$ is denoted in \eqref{S-F0} and here $u$ is given in \eqref{ex-uv}.
Then, we compute the scalar curvatures and they have the following values
\begin{equation*}\label{ex-btt*-G0}
\begin{array}{l}
  \bar\tau = 4(n-1)e^{-4u}\left\{\sin{4v} \tr{S^*}-\cos{4v} \tr{S}\right\},\\[4pt]
  \tilde{\bar{\tau}} =  4(n-1)e^{-4u} \left\{\cos{4v}\tr{S^*}+\sin{4v}\tr{S}\right\},
\end{array}
\end{equation*}
where $u$ and $v$ are given in \eqref{ex-uv}, and
\begin{equation}\label{ex-trStrS*}
\begin{split}
  \tr S &= -2(n-1)\sum_{i=1}^n \frac{\left(x^i\right)^2-\left(y^i\right)^2}{\left[\left(x^i\right)^2-\left(y^i\right)^2\right]^2},\qquad 
  \tr S^* = -4(n-1)\sum_{i=1}^n \frac{x^iy^i}{\left[\left(x^i\right)^2-\left(y^i\right)^2\right]^2}.
\end{split}
\end{equation}
This result supports \thmref{thm:F0F0-tau}.

Bearing in mind \eqref{ex-trStrS*}, we obtain vanishing scalar curvatures $\bar\tau$ and $\tilde{\bar{\tau}}$ for $n=1$, \ie $(\R^{3}, \A\f, \A\bar\xi, \A\bar\eta, \A\bg)$ is scalar-flat.

%
%

Since the two considered $\F_0$-manifolds are related by a transformation from $G_0$,
$(\R^{2n+1}, \A\f, \A\xi, \A\eta, g)$ is flat and the Bochner curvature tensor is an invariant of $G_0$ for dimension at least 7, we
deduce that $B(\bar R)=0$.

Then, bearing in mind \eqref{BR}, the curvature tensor $\bar R$ has an expression corresponding to \eqref{L-RbarR-F0} with  \eqref{L}, namely
\begin{gather}\nonumber
   \bar R = - \bar g\owedge  \bar L +  \bar g^* \owedge  \bar L^* + \left( \bar \eta\otimes \bar \eta\right)\owedge \bar  L,
\\[4pt]\nonumber
 \bar L=\frac{1}{2(n-2)}\bar \rho
-\frac{1}{8(n-1)(n-2)}\left\{\bar \tau\left[\bar  g-  \bar \eta\otimes \bar \eta\right]
+\tilde{\bar {\tau}}\left[\tilde{\bar {g}}-  \bar \eta\otimes\bar  \eta\right]
\right\}.
\end{gather}

Let us recall \cite{Man4}, if $\ell$ is a $G$-transformation determined by \eqref{cct} for functions $(u,v,w)$,
then its inverse transformation $\ell^{-1}$ is the $G$-transformation determined for the functions $(-u,-v,-w)$.
Then, the present example is in unison with \corref{cor:aRsF0-F0} and \corref{cor:aRsF0-F0-LS}.


\section*{Conclusions}

Steady-state solutions of geometric flows, including almost Riemann solitons, are still the subject of intense research and interest in differential geometry.
The present paper introduced almost Riemann solitons on almost contact complex Riemannian manifolds and achieved first results in the study the coexistence of these structures on an odd-dimensional manifold.
More precisely, the most important curvature properties of manifolds obtained from manifolds of cosymplectic type by conformal transformations of the considered structures were described.
Since the study of Riemann solitons is still in its beginning, any contribution in this direction may bring new perspectives on the geometry of the manifold.



\end{document}